\documentclass{amsart}
\usepackage{amsmath,amsthm}
\usepackage{amsfonts,amssymb}
\usepackage{enumerate}
\usepackage{accents,color}
\usepackage{graphicx}

\newcommand{\lvt}{\left|\kern-1.35pt\left|\kern-1.3pt\left|}
\newcommand{\rvt}{\right|\kern-1.3pt\right|\kern-1.35pt\right|}

\newtheorem{thm}{Theorem}[section]
\newtheorem{cor}[thm]{Corollary}
\newtheorem{lem}[thm]{Lemma}
\newtheorem{prop}[thm]{Proposition}

\theoremstyle{remark}

 \def\a{{\alpha}}
 \def\b{{\beta}}
 \def\g{{\gamma}}
 \def\k{{\kappa}}
 \def\t{{\theta}}
 \def\l{{\lambda}}
 \def\d{{\delta}}
 
 \def\s{\sigma}
 \def\la{{\langle}}
 \def\ra{{\rangle}}

 \def\CD{{\mathcal D}}
 
 \def\CH{{\mathcal H}}

 \def\CP{{\mathcal P}}

 \def\CV{{\mathcal V}}

 \def\BB{{\mathbb B}}

 \def\RR{{\mathbb R}}
 \def\SS{{\mathbb S}}
 \def\TT{{\mathbb T}}
 \def\ZZ{{\mathbb Z}}

 \def\sph{{\SS^{d-1}}}
 
 \def\proj{\operatorname{proj}}

\newcommand{\wh}{\widehat}

\def\f{\frac}

\graphicspath{{./}}

\begin{document}

\title[Orthogonal expansions on the unit ball]
 {Orthogonal expansions for generalized Gegenbauer weight function on the unit ball}

\author{Yuan Xu}
\address{Department of Mathematics\\ University of Oregon\\
    Eugene, Oregon 97403-1222.}\email{yuan@uoregon.edu}

\date{\today}
\thanks{The work was supported in part by NSF Grant DMS-1106113}
\keywords{Gegenbauer polynomials, orthogonal polynomials, several variables, reproducing kernel}
\subjclass[2000]{33C45, 33C50, 42C10}

\begin{abstract}
Orthogonal polynomials and expansions are studied for the weight function 
$h_\k^2(x) \|x\|^{2\nu} (1-\|x\|^2)^{\mu-1/2}$ on the unit ball of $\RR^d$, where $h_\k$ is a reflection 
invariant function, and for related weight function on the simplex of $\RR^d$. A concise formula 
for the reproducing kernels of orthogonal subspaces is derived and used to study summability of the
Fourier orthogonal expansions. 
\end{abstract}

\maketitle

\section{Introduction}
\setcounter{equation}{0}

Fourier orthogonal expansions on the unit ball $\BB^d: = \{x \in \RR^d: \|x\| \le 1\}$ of $\RR^d$ have been
studied intensively in recent years (\cite{DaiX, DX}) for the classical weight function
$$
  W_\mu(x):= (1-\|x\|^2)^{\mu -1/2}, \qquad \mu > -1/2,
$$
and, more generally, for the weight functions
$$
  W_{\k,\mu}(x): = h_\k^2(x) (1-\|x\|^2)^{\mu-1/2}, \qquad \mu > -1/2, 
$$
where $h_\k$ is certain weight function that is invariant under a reflection group. Much of the
progress is based on our understanding of orthogonal structure, encapsulated in the concise formulas for the 
reproducing kernels of orthogonal spaces that are integral kernels of orthogonal projection operators. 
These concise formulas serve as an essential tool for studying orthogonal expansions and allow us to define a meaningful 
convolution structure on the unit ball. As an example, let $\CV_n^d(W_\mu)$ be the space of orthogonal 
polynomials of degree $n$ with respect to $W_\mu$ on $\BB^d$. Then the reproducing kernel 
$P_n(W_\mu; \cdot,\cdot)$ of this space satisfies the relation (\cite{X99})
\begin{equation}\label{eq:kernel-Wmu}
   P_n(W_\mu;x,y) = c_\mu \int_{-1}^1 Z_n^{\mu + \frac{d-1}{2}} \left(\la x,y\ra + \sqrt{1-\|x\|^2}  \sqrt{1-\|y\|^2} \, t \right)
    (1-t^2)^{\mu-1} dt,
\end{equation}
where $x,y\in \BB^d$, $c_\mu$ is a the normalization constant so that the integral is 1 when $n =0$ and $Z_n^\l$ 
is a multiple of the Gegenbauer polynomial $C_n^\l$, defined by 
\begin{equation}\label{eq:Zn}
     Z_n^\l(t) := \frac{n+\l}{\l} C_n^\l(t), \qquad \l > 0, \quad -1 \le t \le 1. 
\end{equation}
The orthogonal structure on the unit ball is closely related to that on the unit sphere, so much so that the results 
on the ball can be deduced from the theory of $h$-harmonics with respect to the reflection group. The concise 
formula \eqref{eq:kernel-Wmu} plays the role of the reproducing kernel (zonal harmonic) for spherical harmonics. 

In the present paper, we consider the weight function of the form
$$
  W_{\k,\mu,\nu}(x): = h_\k^2(x) \|x\|^{2 \nu} (1-\|x\|^2)^{\mu-1/2},  
$$
which we shall call the generalized Gegenbauer weight function on the ball and we shall write $W_{\mu,\nu}: =
 W_{0,\mu,\nu}$
when $h_\k(x) \equiv 1$. The additional factor  $\|x\|^{2 \nu}$, which introduces a singularity at the origin of the 
unit ball, breaks down the connection to the theory of $h$-harmonics. Some properties already established 
for $W_{\k,\mu,0}$ do not extend to the setting of $W_{\k,\mu,\nu}$; for example, orthogonal polynomials for
$W_{\mu,\nu}$ are no longer eigenfunctions of a second order linear differential operators. On the other hand, 
a basis of orthogonal polynomials can still be deduced in polar coordinates and we can still deduce a concise 
formula for the reproducing kernel, based on an integral relation for the Gegenbauer polynomials discovered 
recently in \cite{X14}. The latter was derived for $W_{\mu,\nu}$ in \cite{X14}, it motivates our study here
and opens the possibility of carrying out analysis on the ball with respect to the weight function $W_{\k,\mu,\nu}$. 
Our goal in this paper is to explore what is still possible and what might be amiss. 

There is a close relation between orthogonal structure on the unit ball and the standard simplex 
of $\RR^d$, which allows us to consider orthogonal polynomials and expansions for the weight functions such as 
$$
  U_{\k,\mu,\nu}(x):= \prod_{i=1}^d x_i^{\k_i-\f12} |x|^{\nu} (1-|x|)^{\mu-\f12}, \qquad |x|: = x_1+\ldots + x_d, 
$$
on the simplex $\TT^d =\{x\in \RR^d: x_1 \ge 0, \ldots, x_d\ge 0, |x| \le 1\}$.

The paper is organized as follows. In the next section we recall necessary definitions and study orthogonal 
polynomials with respect to $W_{\k,\mu,\nu}$ on the ball. The concise formula for the reproducing kernel 
and orthogonal expansions are studied in the third section. The orthogonal structure and expansion on the
simplex is studied in the fourth section.
 
\section{Orthogonal polynomials on the unit ball}
\setcounter{equation}{0}
 
We start with the definition of the weight function $h_\k$. Let $G$ be a finite reflection group with a fixed 
positive root system $R_+$. Let $\sigma_v$ denote the reflection along $v \in R_+$, that is, $x \sigma_v 
=  x - 2\langle x, v\rangle/\|v\|^2$ for $x \in \RR^d$, where $\langle \cdot,\cdot \rangle$ denote the usual 
Euclidean inner product of $\RR^d$. Let $\kappa: R_+ \mapsto \RR$  be a multiplicity function defined on $R_+$,
which is a $G$-invariant function, and we assume that $\kappa(v) \ge 0$ for all $v \in R_+$. Then the function
\begin{equation}\label{eq:h-weight}
 h_\kappa(x) = \prod_{v \in R_+} |\langle x, v \rangle|^{\kappa(v)}, 
        \qquad x \in \RR^d,
\end{equation}
is a positive homogeneous $G$-invariant function of order $\gamma_\kappa := \sum_{v \in R_+} \kappa_v$. 
The simplest case is when $G = \ZZ_2^d$ for which 
\begin{equation}\label{eq:h-Z2d}
 h_\kappa(x) = \prod_{i=1}^d |x_i|^{\kappa_i}, \qquad \k_i \ge 0. 
\end{equation} 

We consider orthogonal polynomials for the weight function $W_{\kappa,\nu,\mu}$ on the unit ball
\begin{equation}\label{eq:Wball}
  W_{\kappa,\mu,\nu}(x) = h_\kappa^2(x) \|x\|^{2\nu} (1-\|x\|^2)^{\mu-1/2},  
  \quad \mu > -1/2, \quad \nu + \g_\k + d/2 > 0,
\end{equation}
where $h_\kappa$ is as in \eqref{eq:h-weight}. It is easy to verify, in polar
coordinates, that restrictions on $\mu$ and $\nu$ guarantee that this weight function is integrable on $\BB^d$.
We further denote $W_{\mu,\nu}:= W_{0,\mu,\nu}$ and $W_\mu : = W_{\mu,0}$. With respect to $W_{\k,\mu,\nu}$ 
we define an inner product 
\begin{equation}\label{eq:ipd}
\la f, g\ra_{\k,\mu,\nu}: = b_{\k,\mu,\nu}\int_{\BB^d} f(x) g(x)   W_{\k,\mu,\nu}(x)  dx, 
\end{equation}
where $b_{\k,\mu,\nu}$ is the normalization constant such that $\la 1, 1 \ra_{\k,\mu,\nu} =1$. Let $\Pi_n^d$ denote 
the space of polynomials of degree at most $n$ in $d$ variables.  A polynomial $P \in \Pi_n^d$ of degree $n$ is 
called an orthogonal polynomial with respect to $W_{\k,\mu,\nu}$ if $\la P, Q\ra_{\k,\mu,\nu} =0$ for all polynomials 
$Q \in \Pi_{n-1}^d$.  Let $\CV_n^d(W_{\k,\mu,\nu})$ be the space of orthogonal polynomials with respect to the inner
Áproduct \eqref{eq:ipd}. Then $\dim \Pi_n^d = \binom{n+d-1}{n}$. A basis $\{P_{j,n}\}$ for $\CV_n^d(W_{\k,\mu,\nu})$ 
is called mutually orthogonal if $\la P_{j,n}, P_{k,n}  \ra_{\k,\mu,\nu} =0$ whenever $j \ne k$ and it is called 
orthonormal if, in addition,  $\la P_{j,n}, P_{j,n}  \ra_{\k,\mu,\nu} =1$. There are many different bases for the space
$\CV_n^d(W_{k,\mu,\nu})$. The structure of the weight function suggests a particular mutually orthogonal basis that 
can be constructed explicitly. To state this basis, we need $h$-spherical harmonics defined by Dunkl,  which generalize
ordinary spherical harmonics. 

Associated with $G$ and $\k$, the Dunkl operators, $\CD_1,\ldots, \CD_d$, are first order difference-differential 
operators defined by (\cite{D})
$$
   \CD_i f(x) = \partial_i f(x) + \sum_{v\in R_+} \k(v) \frac{f(x) - f(x\s_v)}{\la x,v\ra} v_i,
$$
where $v = (v_1,\ldots, v_d)$ and $x \s_v:= x - 2 \la x,v\ra v/\|v\|^2$. This family of operators enjoys a remarkable 
commutativity, $\CD_i\CD_j = \CD_j\CD_i$, which leads to the definition of the $h$-Laplacian defined by 
$\Delta_h: = \CD_1^2+\ldots +\CD_d^2$. An $h$-harmonic is a homogeneous polynomial that satisfies 
$\Delta_h = 0$ and its restriction on the unit sphere $\sph$ is called spherical $h$-harmonics, which becomes 
ordinary spherical harmonic when $\k =0$. Let $\CH_n^d(h_\kappa^2)$ be the space of $h$-harmonic polynomials 
of degree $n$. For $n \ne m$, it is known that 
$$
 \la Y_n^h, Y_m^h\ra_\k:= b_\k \int_\sph Y_n^h(x)Y_m^h(x) h_\k^2(x) d\s  = 0, \qquad Y_n \in \CH_n^d(h_\k^2), \quad Y_m \in \CH_m^d(h_\k^2), 
$$
where $d\s$ denotes the surface measure on $\sph$ and $b_\k$ is the normalization constant such that 
$\la 1,1\ra_\k =1$. In polar coordinates, the $h$-Laplacian can be written as 
\begin{equation} \label{eq:Delta-pola}
  \Delta_h = \frac{\partial^2}{\partial r^2} + \frac{2\lambda_\k+1}{r} 
      \frac{\partial}{\partial r} + \frac{1}{r^2} \Delta_{h,0},   \qquad   \l_\k := \g_k + \f{d-2}{2},
\end{equation}
where $r = \|x\|$ and $\Delta_{h,0}$ is the spherical part of the $h$-Laplacian, 
which has $h$-harmonics as eigenfunctions. More precisely, if $Y_n^h \in \CH_n^d(h_\kappa^2)$, then
\begin{equation} \label{eq:h-harmonic-eigen}
  \Delta_{h,0} Y_n^h(x) = -n(n+2\lambda_\k) Y_n^h(x).
\end{equation}
In the case of $h_\k(x) =1$, $\Delta_h$ becomes the ordinary Laplacian and $\Delta_{h,0}$ becomes the 
Laplace-Beltrami operator. 

The $h$-harmonics can be used as building blocks of orthogonal polynomials on the unit ball. 
Let $\s_m^d : = \dim \CH_m^d(h_\k^2)$ and let $\{Y_{\ell,m}^h: 1 \le \ell \le \s_m^d\}$ be an orthonormal basis 
of $\CH_m^d(h_\k^2)$, normalized with respect to $\la \cdot,\cdot\ra_\k$, and let $P_n^{(\a,\b)}(t)$ denote the 
usual Jacobi polynomial of degree $n$. Define 
\begin{equation}\label{eq:OPbasis}
   P_{j,\ell}^n(x) : =   P_{j,\ell}^n(W_{\k,\mu,\nu}; x) =
    P_n^{(\mu-\f12, n-2j+\nu+\l_\k)} (2\|x\|^2-1) Y_{\ell, n-2j}^h(x).
\end{equation}

\begin{prop} \label{prop:basis}
The set $\{P_{j,\ell}^n: 1 \le \ell \le \s_{n-2j}^d, 0 \le j \le n/2\}$ is a mutually orthogonal basis of 
$\CV_n^d(W_{\k,\mu,\nu})$ and the norm of $P_{j,\ell}^n$ is given by 
$$
 \la P_{j,\ell}^n, P_{j,\ell}^n \ra_{\k,\mu,\nu} = 
 \frac{ (\nu + \g_\k +\frac{d}{2})_{n-j} (\mu+\f12)_j (n-j+\nu +\mu+ \g_\k+\f{d-1}{2})}
      { j! (\nu +\mu+ \g_\k+\f{d+1}{2})_{n-j}  (n+\nu +\mu+ \g_\k+\f{d-1}{2})}
       =:   H_j^n, 
$$
where $(a)_n$ denotes the Pochhammer symbol, $(a)_n := a(a+1)\cdots (a+n-1)$.
\end{prop}

\begin{proof}
In polar coordinates, it is easy to see that 
$$
 \la f, g\ra_{\k,\mu,\nu} = (b_{\k,\mu,\nu}/b_\k) \int_0^1  \la f(r \cdot), g(r \cdot) \ra_\k
       r^{d-1 + 2 \g_\k+ 2 \nu}(1-r^2)^{\mu-1/2} dr, 
$$
from which the orthogonality of $P_{j,\ell}^n$ follows from the orthogonality of $h$-spherical harmonics and of 
the Jacobi polynomials. The proof is similar to that of classical orthogonal polynomials for $W_\mu$
on the unit ball, the details can be worked out as in \cite[Prop. 5.2.1]{DX}.
\end{proof}

In the case of  $\nu =0$, the orthogonal polynomials are closely related to the $h$-spherical harmonics associated
with $h_\k^2(x) |x_{d+1}|^{2\mu}$ on the sphere $\SS^d$, so much so that it can be deduced from 
\eqref{eq:h-harmonic-eigen} that the orthogonal polynomials in $\CV_n^d(W_{\k,\mu,0}, \BB^d)$ are eigenfunctions
of a second order differential-difference equation;  more precisely, 
\begin{equation}\label{eq:DE-ball}
     D_{\k,\mu} P = - \eta_n^{\k,\mu} P, \qquad  \forall P \in \CV_n^d(W_{\k,\mu,0}, \BB^d),
\end{equation}
where $\eta_n^{\k,\mu} := n (n+2 \l_k + 2\mu  +1)$ and 
$$
  D_{\k,\mu}: = \Delta_h - \la x ,\nabla\ra^2 - (2 \l_\k+2\mu+1) \la x,\nabla\ra. 
$$
This property plays an important role in the study of Fourier orthogonal
expansions on the unit ball; for example, it allows us to define an analogue of the heat kernel operator. One naturally
asks if there is an extension of this property for the case $\nu \ne 0$. 

For this purpose, it is easier to rewrite the basis in \eqref{eq:OPbasis} in terms of the generalized Gegenbauer
polynomials $C_n^{(a,b)}$,  which are orthogonal polynomials with respect to the weight function $|t|^{b}(1-t^2)^{a-1/2}$
on $[-1,1]$ (see \cite[Section 1.5]{DX}). These polynomials satisfy a difference-differential equation that we record 
below.

\begin{prop}
The Generalized Gegenbauer polynomials $C_n^{(a,b)}$ satisfy the equation
$$
   (1-t)^2 y''(t) -(2a + 2 b + 1) t y'(t) + \frac{2 b}{t} \left( y'(t) - \frac{y(t) - y(-t)}{2 t} \right) + n(n+2 a+ 2b)y(t) =0.
$$
\end{prop}

In polar coordinates $(x_1,x_2) = r (\cos \t, \sin\t)$, the polynomials $r^n C_n^{\k_2,\k_1}(\cos \t)$ are 
$h$-spherical harmonics associated with $|x_1|^{\k_1}|x_2|^{\k_2}$ on $\SS^1$, so that the above proposition
follows from \eqref{eq:h-harmonic-eigen}. It is known that 
$$
   C_{2n}^{(a,b)}(t) = \frac{(a+b)_n}{(b+\frac12)_n} P_n^{(a-1/2,b-1/2)} (2t^2-1),
$$ 
which are even functions and for which the differential-difference equation in the proposition simplifies to 
\begin{equation}\label{eq:DEg-gegen}
   (1- t^2) y'' - (2a + 2 b + 1) t y' + \frac{2 b}{t} y' + n(n+2 a+ 2b)y =0.
\end{equation}

In terms of $C_{2n}^{(\a,b)}$, the basis \eqref{eq:OPbasis} becomes
$$
  P_{\ell,j}^n(W_{\k,\mu,\nu};x) = c(j) C_{2j}^{(\mu, n-2j + \l_k + \nu+\f12)}(\|x\|) Y_{\ell,n-2j}^h(x),
$$
where $c(j)$ is a constant. The differential-difference equation can be verified using the following lemma.

\begin{lem} \label{lem:diff}
Let $g(x) = p(\|x\|)Y_{n-2j}^h$ with $p$ being a polynomial of one variable and $Y_{n-2j}^h \in \CH_{n-2j}$.
In the polar coordinates $x = r \xi$, $\xi \in \sph$ and $r \ge 0$, 
\begin{align*}
  \Delta_h g (x) & = \left[ p''(r) + \frac{2 (n-2j)+2\l_\k+1}{r} p'(r)  \right] Y_{n-2j}^h(x), \\
  \frac{d }{dr} g (x)& = \left[ p'(r) + \frac{n-2j}{r} p(r)  \right] Y_{n-2j}^h(x),\\
  \frac{d^2}{dr^2}g(x) & = \left[p''(r) + \frac{2(n-2j)}{r} p'(r)  
       + \frac{(n-2j)(n-2j-1)}{r^2} p(r) \right] Y_{n-2j}^h(x).
\end{align*}
\end{lem}
 
\begin{proof}
Using the fact that $Y_{n-2j}^h(x) = r^{n-2j} Y_{n-2j}^h(\xi)$, the proof of the first item follows from 
\eqref{eq:Delta-pola} and \eqref{eq:h-harmonic-eigen}. The detail, and the proof of the other two identities,
amounts to a straightforward computation. 
\end{proof}

In the polar coordinates $x = r \xi$, it is easy to verify that $\la x, \nabla \ra = r \frac{d}{dr}$. Hence, 
using the identities in the lemma, we can give a direct proof of \eqref{eq:DE-ball} as follows:
setting $p(r) = C_{2j}^{(\mu, n-2j + \l_\k+\f12)}(r)$ and using \eqref{eq:DEg-gegen}, it is straightforward to 
verify that \eqref{eq:DE-ball} holds for $P_{\ell,j}^n(W_{\k,\mu,\nu};x)$, which establishes the identity for all
elements in $\CV_n^d(W_{\k,\mu,\nu})$ since the terms in \eqref{eq:DE-ball} are independent of $j$. 

For the case $\nu \ne 0$, we need to apply the lemma with $p_j(r) = C_{2j}^{(\mu, n-2j + \l_\k + \nu+\f12)}(r)$.
The same consideration, however, yields the following weaker result:

\begin{prop}
The polynomial $P(x) = P_{\ell,j}^n(W_{\k,\mu,\nu};x)$ in \eqref{eq:OPbasis} satisfies 
\begin{align} \label{eq:DE-bad}
 (\Delta_h & - \la x ,\nabla\ra^2 - (2 \l_\k+2\mu+2\nu+1) \la x,\nabla\ra) P
 \\
  &   +  \frac{2 \nu}{\|x\|^2} (\la x ,\nabla\ra  - (n-2j) ) P  = - n (n+2 \l_k + 2\mu +2\nu +1) P. \notag
\end{align}
\end{prop}

The last term in the left hand side in \eqref{eq:DE-bad}, which can be written as $p'(\|x\|) Y_{n-2j}^h(x)$ with
$p(r) = P_j^{(\mu-1/2, n-2j + \l_k+\nu)}(2r^2-1)$ by the second identity in Lemma \ref{lem:diff}, depends on 
the index $j$ in $P_{\ell,j}^n(W_{\k,\mu,\nu};x)$. This means that the \eqref{eq:DE-bad} works only for 
$P_{\ell,j}^n (W_{\k,\mu,\nu};x)$ but does not work for all elements in $\CV_n^d(W_{\k,\mu,\nu})$. This is 
unfortunate, since the fact that $\CV_n^d(W_{\k,\mu,0})$ satisfies the equation \eqref{eq:DE-ball} is essential
for defining an analogue of the heat kernel operator and for define a $K$-functional, both of which play an 
important role in analysis with respect to the weight function $W_{\k,\mu,0}$. 

We end this section with two relations between orthogonal polynomials that have different 
$\nu$ index. 

\begin{prop}
Let $\l_{\k,\nu,\mu} = \l_\k+ \nu + \mu + \f12$. Then the orthogonal polynomials in \eqref{eq:OPbasis} satisfy the 
relations
\begin{align*}
 & (n+\l_\k+ \nu + \mu + \tfrac12) P_{\ell,j}^n(W_{\k,\mu,\nu};x) \\
  & \qquad   =  (j+\mu - \tfrac12) P_{\ell,j-1}^{n-1}(W_{\k,\mu,\nu+1};x)
   +(n-j+\l_\k+ \nu + \mu + \tfrac12)  P_{\ell,j}^n(W_{\k,\mu,\nu+1};x),  
\end{align*}
and 
\begin{align*}
   & (n+\l_\k+ \nu + \mu + \tfrac32) \|x\|^2 P_{\ell,j}^n(W_{\k,\mu,\nu+1};x)  \\
   & \qquad = (j+ 1)  P_{\ell,j+1}^{n+2}(W_{\k,\mu,\nu};x) + (2n-2j+ \l_\k + \nu +1) P_{\ell,j}^n(W_{\k,\mu,\nu};x).
\end{align*}
\end{prop}

\begin{proof}
Using \eqref{eq:OPbasis}, these two identities follow from the corresponding identities for the Jacobi polynomials
given in \cite[(22.7.16)]{AS} and \cite[(22.7.19)]{AS}.
\end{proof}

\section{Orthogonal expansions on the unit ball}
\setcounter{equation}{0}

With respect to the mutually orthogonal basis $\{P_{j,\ell}^n\}$ in the Proposition \ref{prop:basis}, the Fourier
coefficient $\wh f_{j,\ell}^n$ of $f \in L^2(W_{\k,\mu,\nu},\BB^d)$ is defined by 
$\wh f_j^n:= \la f, P_{j,\ell}^n \ra_{\k,\mu,\nu}$ and the Fourier orthogonal expansion of $f$ is defined by 
$$
  f = \sum_{n=0}^\infty \proj_n^{\k,\mu,\nu} f \quad \hbox{with}\quad \proj_n^{\k,\mu,\nu} f(x) := \sum_{0\le j \le n/2}
        \sum_{\ell=1}^{\s_{n-2j}^d}    H_{j,n}^{-1} \wh f_{j,\ell}^n P_{j,\ell}^n(x). 
$$
The projection operator $\proj_n^{\k,\mu,\nu}: L^2(W_{\k,\mu,\nu},\BB^d) \mapsto \CV_n^d(W_{\k,\mu,\nu})$ can 
be written as
$$
  \proj_n^{\k,\mu,\nu} f(x) = b_{\k,\mu,\nu}\int_{\BB^d} f(y) P_n(W_{\k,\mu,\nu}; x,y) W_{\k,\mu,\nu}(y) dy,
$$
where $P_n(W_{\k,\mu,\nu};\cdot,\cdot)$ is the reproducing kernel of $\CV_n^d(W_{\k,\mu,\nu})$ and 
\begin{equation}\label{eq:P-kernel}
  P_n(W_{\k,\mu,\nu}; x,y) :=\sum_{0\le j \le n/2}  \sum_{\ell=1}^{\s_{n-2j}^d} H_{j,n}^{-1} P_{j,\ell}^n(x) P_{j,\ell}^n(y).
\end{equation}
It is known that the reproducing kernel is independent of the choice of orthonormal bases. For the study of 
Fourier orthogonal series, it is essential to obtain a concise formula for the reproducing kernel.  

First we need a concise formula for the reproducing kernel of the $h$-spherical harmonics, for which we 
need the intertwining operator $V_\k$ between the partial derivatives and the Dunkl operators, which
is a linear operator uniquely determined by 
$$
  V_\kappa 1 =1, \qquad V_\kappa \CP_n^d = \CP_n^d, \qquad 
   \CD_i V_\kappa = V_\kappa \partial_i, \quad 1 \le i \le d,
$$
where $\CP_n^d$ is the space of homogeneous polynomials of degree $n$ in $d$ variables. 
The operator $V_\k$ is known to be nonnegative, but the explicit formula of $V_\kappa$ is unknown in general. 
In the case $\ZZ_2^d$, $V_\k$ is an integral operator given by
\begin{equation} \label{eq:Vk}
  V_\k f(x) = c_\k \int_{[-1,1]^d} f(x_1 t_, \ldots, x_d t _d) \prod_{i=1}^d (1+t_i)(1-t_i)^{\k_i-1} dt,
\end{equation}
where $c_\k = \prod_{i=1}^d c_{\k_i}$ and $c_a = \Gamma(a+1/2)/(\sqrt{\pi} \Gamma(a))$ and, if some $\k_i =0$,
the formula holds under the limit 
\begin{equation} \label{eq:limit-0}
  \lim_{a \to 0+} c_a \int_{-1}^1 f(t) (1-t^2)^{a-1} dt = \frac{1}{2} \left [ f(1) + f(-1)\right].
\end{equation}
Let $\{Y_{\ell,n}^h:1 \le \ell \le \s_n^d\}$ be an orthonormal basis of $\CH_n^d(h_\k^2)$. Then the reproducing
kernel of $\CH_n^d(h_\k^2)$ is given by the addition formula of $h$-spherical harmonics, 
\begin{equation}\label{eq:kernel-h}
  \sum_{\ell =1}^{\s_n^d} Y_{\ell,n}^h(x) Y_{\ell,n}^h(y) = V_\k \left [Z_n^{\g_\k+ \frac{d-2}{2}} (\la \cdot ,y\ra ) \right](x),
\end{equation}
where $Z_n^\l$ is a multiple of the Gegenbauer polynomial 
$$
   Z_n^\l(t) := \frac{n+\l}{\l} C_n^\l(t), \qquad -1 \le t \le 1.
$$
For convenience, we define, for given $\k, \nu,  \mu$, 
$$
   \l_{\k,\mu,\nu} := \nu + \mu + \g_\k +\tfrac{d-1}{2}.
$$

\begin{thm} \label{thm:repodBall}
Let $\nu > 0$. If $\mu > 0$, 
\begin{align}\label{eq:reprodBall}
P_n(W_{\k,\mu,\nu}; x,y)   = a_{\k,\mu,\nu}  & \int_{-1}^1 \int_0^1 \int_{-1}^1 
 V_\k \left[ Z_n^{\l_{\k,\mu,\nu}}  (\zeta(\cdot; \|x\|,y, u,v, t)) \right](x') \\
  &   \times  (1-t^2)^{\mu-1} dt u^{\nu-1} (1-u)^{\g_k+\f{d-2}{2}} du (1-v^2)^{\nu-\f12} dv,\notag
\end{align}
where $a_{\k,\mu,\nu}$ is a constant such that the integral is 1 if $n =0$ and 
$$
 \zeta(\cdot; r,y, u,v,t): =r \, \|y\| u v + r \la \cdot, y \ra (1-u) + \sqrt{1-r^2} \sqrt{1-\|y\|^2}\, t ;
$$
furthermore, if $\mu =0$, then the formula holds under the limit \eqref{eq:limit-0}. 
\end{thm}

\begin{proof}
By \eqref{eq:OPbasis}, \eqref{eq:P-kernel} and the addition formula \eqref{eq:kernel-h},
\begin{align*}
  P_n(W_{\k,\mu,\nu}; x,y) = & \sum_{j=0}^{\lfloor \frac{n}{2}\rfloor}  H_{j,n}^{-1} 
     P_{j}^{(\mu-\frac12, \b_{j,n} )}(2 \|x\|^2-1) P_{j}^{(\mu-\frac12, \b_{j,n} )}(2 \|x\|^2-1) \\
      & \times \|x\|^{n-2j} \|y\|^{n-2j} V_\k \left[Z_{n-2j}^{\g_\k+ \frac{d-2}{2}}(\la \cdot,y'\ra)\right](x'), 
\end{align*}
where $\b_{j,n}: = n-2j+\l_{\k,\nu}-\f12$ and $x = \|x\| x'$. The sum in the right hand side is close to
the addition formula for an integral of the Gegenbauer polynomial, except that the index of $Z_{n-2j}^{\l_\k}$ 
does not match. This is where the new integration relation on the Gegenbauer polynomials comes in, which
states, as shown recently in \cite{X14}, that 
\begin{equation*} 
 Z_n^\l(x) = c_\mu \s_{\l+1,\mu} \int_{-1}^1 \int_0^1
    Z_n^{\l+\nu} (u v +(1-u) x) u^{\nu-1} (1-u)^{\l} du\,
   (1-v^2)^{\nu-1/2}dv,
\end{equation*}
where $\l > -1/2$, $\nu > 0$, $\s_{\l,\mu} := \frac{\Gamma( \l+\mu)}{\Gamma(\l) \Gamma(\mu)} $ and $c_\mu:= 
\frac{\Gamma( \mu+1)}{\Gamma(\f12) \Gamma(\mu+\f12)}.$ Using this relation, we can write
\begin{align*}
  P_n(W_{\k,\mu,\nu}; & x,y) =   c_\mu \s_{\l+1,\mu} \int_{-1}^1 \int_0^1  V_\k \Bigg[
   \sum_{j=0}^{\lfloor \frac{n}{2}\rfloor}  H_{j,n}^{-1} 
     P_{j}^{(\mu-\frac12, \b_{j,n}  )}(2 \|x\|^2-1)   \\
          &   \times  P_{j}^{(\mu-\frac12,  \b_{j,n}  )}(2 \|x\|^2-1)
  \|x\|^{n-2j} \|y\|^{n-2j}Z_{n-2j}^{\nu + \g_\k+\frac{d-2}{2}}(s \la \cdot,y'\ra+(1-)y)\Bigg](x') \\
    & \times s^{\nu} (1-s)^{\mu-1} ds\, (1-y^2)^{\nu-1/2}dy.
\end{align*}
This gives the stated result since the sum inside the bracket can be summed up as an integral of the Gegenbauer 
polynomial $Z_n^{\l_{\k,\mu,\nu}}$. This last step is involved but the detail is similar to the proof of Theorem 3.4
in \cite{X14}, where the case $\k = 0$ is established. 
\end{proof}

In the case of $\nu =0$, a concise formula of the reproducing kernel $P_n(W_{\k,\mu,0})$ was established in \cite{X01},
which can be obtained as the limiting case of \eqref{eq:reprodBall} under the limit process of \eqref{eq:limit-0}. The
formula in \cite{X01} was deduced from the concise formula for the reproducing kernels of the $h$-spherical harmonics 
associated with $h_\k^2(x) |x_{d+1}|^{2\mu}$ on the sphere $\SS^d$, which are intimately connected to orthogonal
polynomials with respect to $W_{\k,\mu}$ on $\BB^d$. For $\nu \ne 0$, however, this connection no longer holds. 

In the case of $G = \ZZ_2^d$, the intertwining operator $V_\k$ is given explicitly by \eqref{eq:Vk}. We state this
case as a corollary. 

\begin{cor} \label{cor:kernelZ2d}
Let $W_{\k,\mu,\nu}$ be given in terms of $h_\k$ defined in \eqref{eq:h-Z2d} and let $\nu > 0$. For
$\k_i \ge 0$ and $\nu \ge 0$, 
\begin{align*}
P_n(W_{\k,\mu,\nu}; x,y)  & =  a_{\k,\mu,\nu}  \int_{-1}^1 \int_0^1 \int_{-1}^1 
 \int_{[-1,1]^d } Z_n^{\l_{\k,\mu,\nu}}  (\zeta(x,y, u,v, s, t))  \prod_{i=1}^d (1+s_i)\\
  &   \times \prod_{i=1}^d (1-s_i^2)^{\k_i-1} ds (1-t^2)^{\mu-1} dt u^{\nu-1} (1-u)^{\g_k+\f{d-2}{2}} du (1-v^2)^{\nu-\f12} dv,\notag
\end{align*}
which holds under the limit \eqref{eq:limit-0} when $\mu$ or any $\k_i$ is 0, where
$$
 \zeta( x,y, u,v, s, t): = \|x\| \, \|y\| u v +   (1-u) \sum_{i=1}^d x_i y_i s_i + \sqrt{1-\|x\|^2} \sqrt{1-\|y\|^2}\, t.
$$
\end{cor}

According to these concise formulas, $P_n(W_{\k,\mu,\nu})$ is an integral transform of the Gegenbauer polynomials,
which means that the Fourier orthogonal expansions with respect to $W_{\k,\mu,\nu}$ is connected to the orthogonal
expansions in the Gegenbauer polynomials.  
Let $w_\l(x): = (1-x^2)^{\l-\f12}$ for $\l > -1/2$ and $x \in (-1,1)$ and $c_\l$ be the normalization
constant of $w_\l$. The Gegenbauer polynomials
$C_n^\l$ are orthogonal with respect to $w_\l$.  
For $g \in L^1(w_{\l_{\k,\mu,\nu}}; [-1,1])$ and $x, y \in \BB^d$,  define 
\begin{align}\label{eq:Gg-bound}
L_x^{\k,\mu,\nu} g(y) : = a_{\k,\mu,\nu}  & \int_{-1}^1 \int_0^1 \int_{-1}^1 
 V_\k \left[ g (\zeta(\cdot; x,y, u,v, t)) \right](x') \\
  &   \times  (1-t^2)^{\mu-1} dt u^{\nu-1} (1-u)^{\g_k+\f{d-2}{2}} du (1-v^2)^{\nu-\f12} dv.\notag
\end{align}
For $f \in L^1(W_{\k,\mu,\nu},\BB^d)$ and $g \in L^1(w_{\l_{\k,\mu,\nu}}; [-1,1])$, define
$$
  (f *_{\k,\mu,\nu} g)(x) := b_{\k,\mu,\nu} \int_{\BB^d} f(y) L_x^{\k,\mu,\nu} g(y) W_{\k,\mu,\nu}(y) dy.
$$
This defines a convolution structure with respect to $W_{\k,\mu,\nu}$ on $\BB^d$. To develop its 
property, we start with a lemma. 

\begin{lem}
Let $\nu \ge 0$ and $\mu \ge 0$, and write $\l = \l_{\k,\mu,\nu}$. Then 
for $g \in  L^1(w_\l; [-1,1])$ and $P_n \in \CV_n^d(W_{\k,\mu,\nu})$,  
\begin{equation} \label{eq:IntGx}
  b_{\k,\mu,\nu} \int_{\BB^d} L_x^{\k,\mu,\nu} g(y) P_n(y) W_{\k,\mu,\nu}(y) dy = 
     c_\l \int_{-1}^1 \frac{C_n^\l(t)}{C_n^\l(1)} g(t)w_{\l}(t) dt   P_n (x).
\end{equation}
\end{lem}

\begin{proof}
It follows directly from the definition that 
\begin{equation} \label{eq:kernel-kernel}
   P_n(W_{\k,\mu,\nu};x,y) := L_x^{\k,\mu,\nu} Z_n^{\k,\mu,\nu} (y). 
\end{equation}
If $g$ is a polynomial of degree at most $m$, then $g$ can be written as 
\begin{equation}\label{eq:g-Fourier}
  g(t) = \sum_{k=0}^m \wh g_n^\l Z_k^{\l}  (t), \quad \hbox{with} \quad 
      \wh g_n^\l := c_\mu \int_{-1}^1  \frac{C_k^{\l }(t)}{C_k^{\l}(1)}g(t)w_{\l}(t) dt,
\end{equation}
where we have used the fact that the $L^2$ norm of $C_n^\l$ is equal to $C_n^\l(1) \l /(n+\l)$, which
implies that 
$$
     L_x^{\k,\mu,\nu} g(y) = \sum_{k=0}^n \wh g_n^\l P_k(W_{\k,\mu,\nu};x,y), \qquad x, y \in \BB^d. 
$$
Consequently, if $m \ge n$, then by the definition of the reproducing kernel, 
$$
  b_{\k,\mu,\nu} \int_{\BB^d} L_x^{\k,\mu,\nu} g(y) P_n(y)W_{\k,\mu,\nu}(y) dy =  \wh g_n^\l P_n(y),
$$
which proves \eqref{eq:IntGx} for $g$ being a polynomial of degree $m \ge n$ and, hence, for 
$g \in L^1(w_{\l}; [-1,1])$ by the density of polynomials. 
\end{proof}

\begin{prop} \label{prop:Young}
Let $\nu \ge 0$ and $\mu \ge 0$. Let $p,q,r \ge 1$ and $p^{-1} = r^{-1}+q^{-1}-1$. For $f \in L^q(W_{\k,\mu,\nu},\BB^d)$
and $g \in L^r(w_{\lambda_{\k,\mu,\nu}}; [-1,1])$,
\begin{equation} \label{eq:h-Young}
                \|f *_{\k,\mu,\nu} g\|_{W_{\k,\mu,\nu},p} \le \|f\|_{W_{\k,\mu,\nu},q} \|g\|_{w_{\l_{\k,\mu,\nu}},r}.
\end{equation}
\end{prop}

\begin{proof}
Following the standard proof of Young's inequality, it is sufficient to show that 
$\| L_x^{\k,\mu,\nu} g \|_{W_{\k,\mu,\nu}, r} \le \|g\|_{w_{\k,\mu,\nu}, r}$ for $1 \le r \le \infty$. 
Since $V_\k$ is nonnegative, $|V_\k g| \le V_k(|g|)$, it follow that $| L_x^{\k,\mu,\nu} g | \le 
L_x^{\k,\mu,\nu} (|g|)$.  Hence, the inequality \eqref{eq:Gg-bound} holds for $p = \infty$ directly by the 
definition and for $p =1$ by applying \eqref{eq:IntGx}. The log-convexity of the $L^r$-norm establishes the
case for $1 < r < \infty$. 
\end{proof}

\begin{prop}\label{prop:conv-justify}
Let $\nu, \mu \ge 0$ and let $\wh g_n^\l$ be the Fourier coefficient of $g$ defined in \eqref{eq:g-Fourier}. Then
for $f \in L^1(W_{\k,\mu,\nu},\BB^d)$ and $g \in L^1(w_{\lambda_{\k,\mu,\nu}}; [-1,1])$,
$$
  \proj_n^{\k,\mu,\nu} (f *_{\k,\mu,\nu} g) (x) =  \wh g_n^{\l_{\k,\mu,\nu}} \proj_n^{\k,\mu,\nu} f(x)  
$$
\end{prop}

This proposition justifies calling $*_{\k,\mu,\nu}$ a convolution. Its proof follows easily from \eqref{eq:IntGx} and 
from exchange of integrals. 

For $\d > 0$, the Ces\`aro $(C,\d)$ means $S_n^\d (W_{\k,\mu,\nu};f)$ of the Fourier orthogonal expansion is defined by 
$$
 S_n^\d (W_{\k,\mu,\nu};f) := \f{1}{\binom{n+\d}{d}} \sum_{k=0}^n \binom{n-k+\d}{n-k} \proj_k^{\k,\mu,\nu} f,
$$
which can be written as an integral of $f$ against the kernel  $K_n^\d(W_{\k,\mu,\nu};x,y)$.
Let $k_n^\d(w_\l; s,t)$ be the Ces\`aro  $(C,\delta)$ means of the Gegenbauer series; then
$$
  k_n^\d (w_\l;s,1) = \frac{1}{\binom{n+\d}{n}} \sum_{k=0}^n \binom{n-k+\d}{n-k} Z_k^\l(s). 
$$
As a consequence of  \eqref{eq:kernel-kernel}, we can write
\begin{equation} \label{eq:cesaro}
   K_n^\d (W_{\k,\mu,\nu}; x,y) = L_x \left[k_n^\d(w_{\l_{\k,\mu,\nu}}; \cdot, 1)\right ](y). 
\end{equation}

\begin{thm}
For $\mu,\nu \ge 0$, the Ces\`aro $(C,\delta)$ means for $W_{\k, \mu,\nu}$ satisfy 
\begin{enumerate} [\quad 1.]
\item if $\d \ge 2 \l_{\k, \nu, \mu} + 1$, then $S_n^\d(W_{\k,\mu,\nu}; f) \ge 0$ if $f(x) \ge 0$;
\item $S_n^\d(W_{\k,\mu,\nu}; f)$ converge to $f$ in $L^1(W_{\k,\mu,\nu}; \BB^d)$ norm or $C(\BB^d)$ norm 
if $\d > \l_{\k,\nu,\mu}$. 
\end{enumerate}
\end{thm}

\begin{proof}
The first assertion follows immediately from the non-negativity of the Gegenbauer series \cite{Gas}. For the 
second one, it is sufficient to show that 
$$
   \max_{x \in \BB^d} \int_{\BB^d} |K_n^\d(W_{\k,\mu,\nu}; x,y)| W_{\k,\mu,\nu}(y)dy 
$$
is bounded, which can be deduced easily from the fact that the integral of $|k_n^\d(w_\l; t,1)|$ against 
$w_\l$ is bounded if $\d > \l$ by using \eqref{eq:cesaro} and applying \eqref{eq:IntGx} with $P_n(y) =1$. 
\end{proof}
 
In the case of $\k =0$, it is shown in \cite{X14} that $\d > \nu + \mu + \f{d-1}{2}$ is also necessary for 
the second item in the above theorem. However, for $\nu = 0$, the necessary and sufficient condition is known 
in the case of $G= \ZZ_2^d$ as $\d > \s_{\k,\mu}: = \gamma_\k - \min_{1 \le i \le d} \k_i + \mu + \f{d-1}{2}$
(\cite{DaiX09}), which requires delicate estimate of the $(C, \d)$ kernel based on the explicit formula in Corollary
\ref{cor:kernelZ2d}. We expect that the necessary and sufficient condition for the second item of the theorem
is $\d > \nu + \sigma_{\k,\mu}$. 

We can also define the Poisson integral for $f \in L^1(W_{\k,\mu,\nu}, \BB^d)$ by 
\begin{equation*} 
        P_r (W_{\k,\mu,\nu}; f)  :=   f *_{\k,\mu,\nu} P_r^{\k,\mu,\nu}, 
\end{equation*}
where $0<r<1$ and  the kernel $P_r^{\k,\mu,\nu}$ is defined by 
\begin{equation*} 
      P_r^{\k,\mu,\nu}(x, y) := L_x ^{\k,\mu,\nu} P_r, \qquad P_r(t) = \frac{1 -r^2}{ (1 - 2r t + r^2)^{\l_{\k,\mu,\nu}+1}}.
\end{equation*}
The Poisson kernel is non--negative and it satisifies 
$$
  P_r^{\k,\mu,\nu}(x,y) =  \sum_{n=0}^\infty  P_n (W_{\k,\mu,\nu}; x,y) r^n, \qquad 0 < r < 1. 
$$
The standard proof for the Poisson integral of orthogonal expansions leads to: 
 
\begin{thm} \label{thm:h-poisson}
For $f \in L^p( W_{\k,\mu,\nu}, \BB^d )$ if $1 \le p < \infty$, or $f \in C(\BB^d)$ if $p = \infty$,
$\lim_{r \to 1-} \|P_r(W_{\k,\mu,\nu}; f) - f\|_{W_{\k,\mu,\nu},p} = 0$.
\end{thm}

\section{Orthogonal polynomials and expansions on the simplex}
\setcounter{equation}{0}

There is a close relation between orthogonal polynomials on the unit ball $\BB^d$ and those on the simplex 
$$
 \TT^d: = \{x \in \RR^d: x_1 \ge 0, \ldots, x_d \ge 0, 1- |x| \ge 0\}, \quad |x|: = x_1+\ldots + x_d,
$$
under the mapping $\psi : x \in \BB^d \mapsto (x_1^2,\ldots, x_d^2) \in \TT^d$. Assume that $h_\k$ is the
reflection invariant weight function in \eqref{eq:h-weight} that is also invariant under $\ZZ_2^d$, which
means that the reflection group $G$ is a semi-product of a reflection group $G_0$ and $\ZZ_2^d$. Associated 
to this weight function, we define a weight function $U_{\k,\mu,\nu}$ on the simplex $\TT^d$ by 
\begin{equation} \label{eq:weightTh}
  U_{\k,\mu,\nu}(x) = h_\k(\sqrt{x_1}, \ldots, \sqrt{x_d}) |x|^{\nu} (1-|x|)^{\mu-1/2}, \quad \nu + \g_\k+ d/2 >0,\, \mu > -1/2,
\end{equation}
which means that $W_{\k,\mu,\nu}(x) = (U_{\k,\mu,\nu} \circ \psi)(x) |x_1\cdots x_d|$, where $W_{\k,\mu,\nu}$ is the 
weight function in \eqref{eq:Wball} on $\BB^d$. In the case of $h_\k$ 
in \eqref{eq:h-Z2d} associated to the group $\ZZ_2^d$, the weight function is 
\begin{equation} \label{eq:weightT}
   U_{\k,\mu,\nu}(x) =\prod_{i=1}^d x_i^{\k_i-1/2} |x|^{\nu} (1-|x|)^{\mu-1/2}, \quad \k_i \ge 0,
\end{equation}
which is the classical Jacobi weight function when $\nu =0$. The case $\nu \ne 0$ has not been considered up
to now. 

With respect to $U_{\k,\mu,\nu}$ we define the inner product on $T^d$ by
$$
   \la f,g\ra_{\k,\mu,\nu}^T: = b_{\k,\nu,\mu} \int_{\TT^d} f(x) g(x)  U_{\k,\mu,\nu}(x) dx. 
$$
Let $\CV_n^d(U_{\k,\mu,\nu}, \TT^d)$ be the space of orthogonal polynomials with respect to this inner product. 
It can be shown, as in the case of $\nu =0$ (cf. \cite[Sect. 4.4]{DX}), that $\la f,g\ra_{\k,\mu,\nu}^T = 
 \la f\circ \psi,g \circ \psi \ra_{\k,\mu,\nu}$, where $ \la f,g\ra_{\k,\mu,\nu}$ 
is the inner product on $\BB^d$ defined in \eqref{eq:ipd} and, as a consequence, $\psi$ induces a one-to-one
correspondence between $P \in \CV_n^d(U_{\k,\mu,\nu}, \TT^d)$ and $P \circ \psi \in G\CV_{2n}^d(W_{\k,\mu,\nu})$,
the subspace of $\CV_n^d(W_{\k,\mu,\nu})$ that contains polynomials invariant under $\ZZ_2^d$.  In particular, 
let $G\CH_m^d(h_\k^2)$ be the space that contains $h$-spherical harmonics in $\CH_m^d(h_\k^2)$ that are 
invariant under $\ZZ_2^d$. 

\begin{prop}
For $0 \le j \le n$, let $\{Y_{\ell,2n-2j}^h: 1 \le \ell \le \binom{n-j+d-1}{n-j}\}$ be an orthonormal basis of 
$G\CH_{2n-2j}^d(h_\k^2)$. Define 
\begin{equation}\label{eq:OPbasisT}
   P_{j,\ell}^n(U_{\k,\mu,\nu}; x) :=
          P_n^{(\mu-\f12, n-2j+\nu+\l_\k)} (2|x|^2-1) Y_{\ell, 2n-2j}^h(\sqrt{x_1},\ldots,\sqrt{x_d}).
\end{equation}
Then the set $\{P_{j,\ell}^n(U_{\k,\mu,\nu}): 1 \le \ell \le \s_{n-2j}^d, 0 \le j \le n\}$ is a mutually orthogonal basis of 
$\CV_n^d(U_{\k,\mu,\nu}, \TT^d)$.
\end{prop}

The mapping between orthogonal polynomials on the unit ball and those on the simplex extends to the 
reproducing kernels for the respective spaces, which allows us to derive a concise formula for the reproducing
kernel $P_n(U_{\k,\mu,\nu};\cdot,\cdot)$ of $\CV_n^d(U_{\k,\mu,\nu};\TT^d)$, defined similarly as the
one on the unit ball, and $P_n(U_{\k,\mu,\nu})$ is the kernel function for the projection operator 
$\proj_{n,\TT}^{\k,\mu,\nu}: L^2(U_{\k,\mu,\nu}, \TT^d) \mapsto \CV_n^d(U_{\k,\mu,\nu})$.  Indeed, for all 
$x,y \in \TT^d$,  it is known (\cite[Thm. 4.4.5]{DX}) that 
\begin{equation} \label{eq:PnB=T}
  P_n(U_{\k,\mu,\nu};x,y) = 2^{-d} \sum_{\varepsilon \in \ZZ_2^d} P_{2n}(W_{\k,\mu,\nu};\sqrt{x}, \varepsilon \sqrt{y}),
\end{equation}
where $\sqrt{x} = (\sqrt{x_1},\ldots, \sqrt{x_d})$ and $\varepsilon u = (\varepsilon_1 u_1 , \ldots, \varepsilon_d u_d)$. 
This identity allows us to deduce a concise formula for $P_n(U_{\k,\mu,\nu};x,y)$ from
Theorem \ref{thm:repodBall}, in terms of the Gegenbauer polynomial $Z_{2n}^{\l_{\k\nu,\mu}}$, which we can rewrite
in terms of the Jacobi polynomial $P_n^{(\l_{\k,\mu,\nu}, -1/2)}$ by the quadratic transform between these two 
polynomials, that is, 
$$
  Z_{2n}^\l(t) = \frac{2n+\lambda}{\lambda} C_{2n}^{\lambda}(t) =  
p_n^{(\lambda-\frac12,-\frac12)}(1) p_n^{(\lambda-\frac12,-\frac12)}(2 t^2 -1)=: \Xi_n^\l(2t^2-1),   
$$ 
where $p_n^{(a,b)}$ denote the orthonormal Jacobi polynomial of degree $n$. We state this formula explicitly in the 
case of $G = \ZZ_2^d$. Recall that $\l_{\k,\mu,\nu} = \nu + \mu + \g_\k +\tfrac{d-1}{2}$.

\begin{thm} \label{thm:reprodSimplex}
Let $W_{\k,\mu,\nu}$ be given in terms of $h_\k$ defined in \eqref{eq:h-Z2d} and let $\nu > 0$. For
$\k_i \ge 0$ and $\nu \ge 0$, 
\begin{align*}
P_n(U_{\k,\mu,\nu}; x,y)  & =  a_{\k,\mu,\nu}  \int_{-1}^1 \int_0^1 \int_{-1}^1 
 \int_{[-1,1]^d } \Xi_n^{\l_{\k,\mu,\nu}}  (2 \zeta(x,y, u,v, s, t)^2-1) \\
 &   \times \prod_{i=1}^d (1-s_i^2)^{\k_i-1} ds (1-t^2)^{\mu-1} dt u^{\nu-1} (1-u)^{\g_k+\f{d-2}{2}} du (1-v^2)^{\nu-\f12} dv,\notag
\end{align*}
which holds under the limit \eqref{eq:limit-0} when $\mu$ or any $\k_i$ is 0, where
$$
 \zeta( x,y, u,v, s, t): = \sqrt{|x|} \sqrt{|y|} u v +   (1-u) \sum_{i=1}^d \sqrt{x_i y_i} s_i + \sqrt{1-|x|} \sqrt{1-|y|^2}\, t.
$$
\end{thm}
 
In the case $\nu =0$, this formula and its version for more general $h_\k$ are known (cf. \cite{X01}); the case 
$\nu \ne 0$ is new. We can also define a convolution $*^\TT_{\k,\mu,\nu}$ between $f \in L^1(U_{\k,\mu,\nu}; \TT^d)$
and $g \in L^1(w_{\l_{\k,\mu,\nu} -\f12, -\f12}, [-1,1])$, where $w_{a,\b}(t) : = (1-t)^a (1+t)^b$. In fact, it can be defined 
as follows: 
$$
(f *^\TT_{\k,\mu,\nu} g \circ \psi )(x) := (f\circ \psi) *_{\k,\mu,\nu} g(2 \{\cdot \}^2-1) (x), 
$$
where the convolution in the right hand side is the one defined in Section 3. The properties of this convolution can 
then be deduced from the corresponding results on the unit ball. In particular,  Proposition \ref{prop:Young} holds
with the norm of $\|\cdot\|_{U_{\k,\mu,\nu},p}$ and $\|\cdot \|_{w_{\l_{\k,\mu,\nu} -\f12, -\f12}, p}$. Much of the 
analysis from this point on can be carried out from the correspondence between analysis on the ball and on the
simplex, just as in the case of $\nu = 0$. We conclude this section with a result on summability. 
 
Let $S_n^\d(U_{\k,\mu,\nu}; f)$ be the Ces\`aro $(C,\delta)$ means of the Fourier orthogonal expansion with
respect to $U_{\k,\mu,\nu}$ on $\TT^d$ and let $K_n^\d(U_{\k,\mu,\nu}; \cdot,\cdot)$ be its kernel, both are 
defined similarly as the corresponding ones on the unit ball. In particular, we can also write
$$
  S_n^\d(U_{\k,\mu,\nu}; f) = f *^\TT_{\k,\mu,\nu} k_n^\d (w_{\l_{k,\mu,\nu} -\f12, \f12}),
$$
where $k_n^\d (w_{a,b}; s, t) = k_n^\d (w_{a,b}; s, 1)$ denotes the Ces\`aro $(C,\d)$ kernel of the Jacobi series 
for $w_{a,b}$ on $[-1,1]$ with one variable evaluated at $1$.  

\begin{thm}
For $\l \ge 0$ and $\mu \ge 0$, the Ces\`aro $(C,\delta)$ means for $U_{\k, \l,\mu}$ satisfy 
\begin{enumerate} [\quad 1.]
\item if $\d \ge 2 \l_{\k,\mu,\nu} +1$, then $S_n^\d(U_{\k,\mu,\nu}; f) \ge 0$ if $f(x) \ge 0$;
\item $S_n^\d(U_{\k,\mu,\nu}; f)$ converge to $f$ in $L^1(U_{\k,\mu,\nu}; \TT^d)$ norm or $C(\TT^d)$ norm 
if $\d > \l_{\k,\nu,\mu}$. 
\end{enumerate}
\end{thm}
 
We can also define the Poisson integral and establish an analogue of Theorem \ref{thm:h-poisson}.

\end{document}